\documentclass[11pt, reqno]{amsart}

\usepackage{amsmath, amsthm, amscd, amsfonts, amssymb, graphicx, color}
\usepackage[bookmarksnumbered, colorlinks, plainpages]{hyperref}
\input{mathrsfs.sty}
\hypersetup{colorlinks=true,linkcolor=red, anchorcolor=green, citecolor=cyan, urlcolor=red, filecolor=magenta, pdftoolbar=true}

\newtheorem{theorem}{Theorem}[section]
\newtheorem{lemma}[theorem]{Lemma}
\newtheorem{proposition}[theorem]{Proposition}
\newtheorem{corollary}[theorem]{Corollary}
\theoremstyle{definition}

\newtheorem{example}[theorem]{Example}

\theoremstyle{remark}
\newtheorem{remark}[theorem]{Remark}
\numberwithin{equation}{section}

\begin{document}

\title[Norm-parallelism in the geometry of Hilbert $C^*$-modules]{Norm-parallelism in the geometry of Hilbert $C^*$-modules}

\author[A. Zamani, M.S. Moslehian]{Ali Zamani and Mohammad Sal Moslehian$^*$}

\address{Department of Pure Mathematics, Center of Excellence in Analysis on Algebraic Structures (CEAAS), Ferdowsi University of
Mashhad, P.O. Box 1159, Mashhad 91775, Iran}
\email{moslehian@um.ac.ir; moslehian@member.ams.org; zamani.ali85@yahoo.com}



\begin{abstract}
Utilizing the Birkhoff--James orthogonality, we present some characterizations of the norm-parallelism for elements of $\mathbb{B}(\mathscr{H})$ defined on a finite dimensional Hilbert space, elements of a Hilbert $C^*$-module over the $C^*$-algebra of compact operators and elements of an arbitrary $C^*$-algebra. We also consider the characterization of norm parallelism problem for operators on a finite dimensional Hilbert space when the operator norm is replaced by the Schatten $p$-norm. Some applications and generalizations are discussed for certain elements of a Hilbert $C^*$-module.

\end{abstract} \maketitle

\section{Introduction and preliminaries}

Let $\mathbb{B}(\mathscr{H})$ and $\mathbb{K}(\mathscr{H})$ be the $C^*$-algebras of all bounded operators and compact operators on a complex Hilbert space $\mathscr{H}$ endowed with an inner product $[., .]$, respectively. A compact operator $T$ is said to be in the Schatten $p$-class $\mathcal{C}_p\,\,(1\leq p <\infty)$ if ${\rm tr}(|T|^p)< \infty$, where tr denotes the usual trace and the symbol $|T|$ stands for $(T^\ast T)^\frac{1}{2}$. The Schatten $p$-norm of $T$ is defined by $\|T\|_p = ({\rm tr}(|T|^p))^{\frac{1}{p}}$. If $1< p <\infty$, then the norm $\|.\|_p$ is Fr\'{e}chet differentiable at every $T$. In this case
\begin{align*}
\frac{d}{dt}{\Big|}_{t = 0} \|T + tS\|^p_p = p\,\mbox{Re}\,{\rm tr}(|T|^{p - 1}U^*S),
\end{align*}
for every $S$, where $T = U|T|$ is the polar decomposition of $T$. In \cite[Theorem 2.1]{A.E.G}, $\|.\|_1$ is Fr\'{e}chet differentiable at invertible
$T$ if $\mathscr{H}$ is finite dimensional.
The symbol $I$ stands for the identity operator on $\mathscr{H}$.

Recall that, a \emph{positive} element $a$ of a $C^{*}$-algebra $\mathscr A$ is a self-adjoint element whose spectrum
$\sigma(a)$ is contained in $[0, \infty)$. If $a\in\mathscr A$ is positive, we write $a\geq0$. Furthermore, if $a\in\mathscr A$ is positive, then there exists a unique positive $b\in\mathscr A$ such that $a=b^2$. Such an element $b$ is called the \emph{positive square root} of $a$ and is denoted by $a^{\frac{1}{2}}$, in particular $|a| = (a^\ast a)^\frac{1}{2}$. A linear functional $\varphi$ over $\mathscr A$ of norm one is called a \emph{state} if $\varphi(a)\geq0$ for any positive element $a\in\mathscr A$. Also, $r(a)$ stands for the spectral radius of an arbitrary element $a\in\mathscr A$. More details on the theory of $C^*$-algebras can be found e.g. in \cite{mor}.

Hilbert $C^{*}$-modules are essentially objects like Hilbert spaces, except that the inner product, instead of being complex-valued, takes its values in a $C^{*}$-algebra. Although Hilbert $C^{*}$-modules behave like Hilbert spaces in some ways, some fundamental Hilbert space properties
like Pythagoras' equality, the adjointability of operators, and decomposition into orthogonal complements do not hold in general.
An inner product $C^*$-module over a $C^{*}$-algebra $\mathscr A$ is a complex linear space $\mathscr{X}$ which is a right $\mathscr A$-module
with a compatible scalar multiplication (i.e., $\gamma(xa) = (\gamma x)a = x(\gamma a)$ for all $x\in \mathscr{X}, a\in\mathscr A, \gamma\in\mathbb{C}$) and equipped with an $\mathscr A$-valued inner product $\langle\cdot,\cdot\rangle \,: \mathscr{X}\times \mathscr{X}\longrightarrow\mathscr A$ satisfying\\
(i) $\langle x, \gamma y + \mu z\rangle = \gamma\langle x, y\rangle + \mu\langle x, z\rangle$,\\
(ii) $\langle x, ya\rangle=\langle x, y\rangle a$,\\
(iii) $\langle x, y\rangle^*=\langle y, x\rangle$,\\
(iv) $\langle x, x\rangle\geq0$ and $\langle x, x\rangle=0$ if and only if $x=0$,\\
for all $x, y, z\in \mathscr{X}, a\in\mathscr A, \gamma, \mu\in\mathbb{C}$.
It is easy to observe that $\|x\|=\|\langle x, x\rangle\|^{\frac{1}{2}}$ defines a norm on $\mathscr{X}$. If $\mathscr{X}$ with respect to this norm is complete, then it is called a \emph{Hilbert $\mathscr A$-module}, or a \emph{Hilbert $C^*$-module} over $\mathscr A$. Complex Hilbert spaces can be regarded as Hilbert $\mathbb{C}$-modules. Any $C^*$-algebra $\mathscr A$ is indeed a Hilbert $C^*$-module over itself via $\langle a, b\rangle :=a^* b$. For every $x\in \mathscr{X}$ the positive square root of $\langle x, x\rangle$ is denoted by $|x|$. Let us define \emph{elementary operator} $\theta_{x,y}$ by the formula $\theta_{x,y}(z) = x\langle y, z\rangle \,\, (x, y, z\in\mathscr{X})$. It is
easy to see that $\|\theta_{x,y}\| = \|x\|\,\|y\|$, ${\theta}^\ast_{x,y} = \theta_{y,x}$ and $\theta_{x,y}\theta_{z,w} = \theta_{x\langle y, z\rangle,w} \,\, (z, w\in\mathscr{X})$.

In an inner product $\mathcal A$-module $\mathscr{X}$ we have
$${\langle x, y\rangle}^*\langle x, y\rangle \leq \|\langle x, x\rangle\|\langle y, y\rangle \qquad(x, y\in \mathscr{X}),$$
wherefrom the following Cauchy--Schwarz inequality holds (see also \cite{F.F.M.S}):
$$\|\langle x, y\rangle\|^2\leq\|\langle x, x\rangle\|\,\|\langle y, y\rangle\|\qquad(x, y\in \mathscr{X}).$$
Furthermore, if $\varphi$ is a state over $\mathscr A$, then $(x, y)\mapsto\varphi(\langle x, y\rangle)$ gives rise to a usual semi-inner product on $\mathscr{X}$, so we have the following useful Cauchy–-Schwarz inequality:
$$|\varphi(\langle x, y\rangle)|^2\leq\varphi(\langle x, x\rangle)\varphi(\langle y, y\rangle)\qquad(x, y\in \mathscr{X}).$$
We refer the reader to \cite{lan, M.T} for more information on the basic theory of Hilbert $C^*$-modules.

Throughout this paper we assume that $\mathscr A$ is an arbitrary $C^*$-algebra (not necessarily unital), $(\mathscr{X}, \langle ., .\rangle)$ is a Hilbert $\mathscr A$-module and $(\mathscr{H}, [., .])$ denotes a Hilbert space.

If $\mathscr{V}$ is a normed space and $x, y\in \mathscr{V}$, we say that $x$ is \emph{norm-parallel} to $y$, denoted by $x\parallel y$ if
$$\|x+\lambda y\|=\|x\|+\|y\|$$
for some $\lambda\in\mathbb{T}=\{\alpha\in\mathbb{C}: \,\,|\alpha|=1\}$; see \cite{S}. Clearly, two elements of a Hilbert space are norm-parallel if they are linearly dependent.

Notice that the norm-parallelism is symmetric (i.e., $x\parallel y\, \Leftrightarrow \, y\parallel x$ ) and homogeneous (i.e., $x\parallel y\, \Leftrightarrow \, \alpha x\parallel \beta y \quad \mbox{for all}\,\, \alpha, \beta\in\mathbb{R}$) (see \cite{M.Z}), but not transitive (i.e., $x\parallel y$ and $y\parallel z$ $\nRightarrow$ $x\parallel z$) (see Example \ref{ex.19}).

Some characterizations of the norm-parallelism for elements of $C^*$-algebras and Hilbert $C^*$-modules were given in \cite{M.Z}.

Recall that, an element $x\in\mathscr{V}$ is said to be the \emph{Birkhoff--James} orthogonal to another element $y\in\mathscr{V}$, in short $x\perp_{BJ} y$,
if $\|x\|\leq \|x + \gamma y\|$ for all $\gamma\in\mathbb{C}$. It is easy to see that Birkhoff--James orthogonality is equivalent to the usual
orthogonality in case $\mathscr{V}$ is a Hilbert space. Orthogonality of matrices and some distance problems are presented in \cite{B.S}. When $\mathscr{V}$ is a Hilbert $C^*$-module, some interesting characterizations of Birkhoff--James orthogonality were given by Aramba\v{s}i\'{c} and Raji\'{c} \cite{A.R.2}; see also \cite{A.R.1}. The same result is later obtained in \cite{B.G} by using a different approach.

The purpose of this paper is to investigate the norm-parallelism in the setting of Hilbert $C^{*}$-modules and present some relationship between norm-parallelism and Birkhoff--James orthogonality. We give some characterizations of the norm-parallelism for elements of $\mathbb{B}(\mathscr{H})$ defined on a finite dimensional Hilbert space and for elements of a $C^*$-algebra. In addition, we show that for $T\in\mathbb{B}(\mathscr{H})$ if the norm attaining
set $M_T$ is a unit sphere of some finite dimensional
subspace $\mathscr{H}_0$ and $\|T\|_{{\mathscr{H}_0}^\perp} < \|T\|$, then $T\parallel S$ implies that there exists a unit vector $\xi$ such that $|[T\xi, S\xi]| = \|T\|\,\|S\|$, i.e., $T\xi$ and $S\xi$ are linearly dependent. Among other things, we consider the characterization of norm parallelism problem for operators on a finite dimensional Hilbert space when the operator norm is replaced by the Schatten $p$-norm.

\section{Results}
The characterization of the norm-parallelism for elements of a
Hilbert $C^*$-module was obtained
in \cite{M.Z}. The following result is a combination of Theorems 4.1 and 4.6 of \cite{M.Z}.
\begin{lemma}\cite{M.Z}\label{lemma.280}
For $x , y\in \mathscr{X}$ the following statements are equivalent:
\begin{itemize}
\item[(i)] $x\parallel y$.
\item[(ii)] There exists a state $\varphi$ over $\mathscr A$ such that $|\varphi({\langle x, y\rangle})| = \|x\|\|y\|$.
\item[(iii)] ${\langle x, x\rangle}\parallel{\langle x, y\rangle}$ and $\|{\langle x, y\rangle}\|=\|x\|\|y\|$.
\item[(iv)] $r({\langle x, y\rangle})=\|{\langle x, y\rangle}\|=\|x\|\|y\|$.
\end{itemize}
\end{lemma}
In the case of a Hilbert $C^*$-module over the $C^*$-algebra $\mathbb{K}(\mathscr{H})$ of compact operators, we have the following result.
\begin{proposition}\label{pr.03.5}
Let $\mathscr{X}$ be a Hilbert $\mathbb{K}(\mathscr{H})$-module. For $x , y\in \mathscr{X}$ the following statements are equivalent:
\begin{itemize}
\item[(i)] $x\parallel y$.
\item[(ii)] There exists a positive operator $P\in\mathcal{C}_1(\mathscr{H})$ of trace one such that $|{\rm tr}(P{\langle x, y\rangle})| = \|x\|\|y\|$.
\end{itemize}
\end{proposition}
\begin{proof}
(i)$\Rightarrow$(ii) Let $x\parallel y$. By the equivalence $(i) \Leftrightarrow (ii)$ of Lemma \ref{lemma.280}, there exists a state $\varphi$ over $\mathbb{K}(\mathscr{H})$ such that $|\varphi({\langle x, y\rangle})| = \|x\|\|y\|$. By \cite[Theorem 4.2.1]{mor}, there exists a positive operator $P\in\mathcal{C}_1(\mathscr{H})$ of trace one such that $\varphi(T) = {\rm tr}(PT)$, $T\in\mathbb{K}(\mathscr{H})$. Thus we have $|{\rm tr}(P{\langle x, y\rangle})| = \|x\|\|y\|$.\\
(ii)$\Rightarrow$(i) Suppose that there exists a positive trace one operator $P\in\mathcal{C}_1(\mathscr{H})$ such that $|{\rm tr}(P{\langle x, y\rangle})| = \|x\|\|y\|$.
Define the state $\varphi$ over $\mathbb{K}(\mathscr{H})$ by $$\varphi(T) = {\rm tr}(PT)\qquad(T\in \mathbb{K}(\mathscr{H})).$$
We therefore have $|\varphi({\langle x, y\rangle})| = \|x\|\|y\|$. It follows from the equivalence $(i) \Leftrightarrow (ii)$ of Lemma \ref{lemma.280} that $x\parallel y$.
\end{proof}
We have the following characterization of the norm-parallelism for elements of a Hilbert $C^*$-module.
\begin{theorem}\label{th.05}
For $x , y\in \mathscr{X}$ the following statements are equivalent:
\begin{itemize}
\item[(i)] $x\parallel y$.
\item[(ii)] There exist a sequence of unit vectors $\{\xi_n\}$ in a Hilbert space $\mathscr{H}$ and $\lambda\in\mathbb{T}$ such that
$$\lim_{n\rightarrow\infty} \mbox{Re}[\langle x, \lambda y\rangle\xi_n, \xi_n]= \|x\|\,\|y\|.$$
\end{itemize}
\end{theorem}
\begin{proof}
(i)$\Rightarrow$(ii) Let $x\parallel y$. Hence $\|x + \lambda y\| = \|x\| + \|y\|$ for some $\lambda\in\mathbb{T}$.
By the Gelfand--Naimark theorem we can regard $\mathscr A$ as a $C^*$-subalgebra of $\mathbb{B}(\mathscr{H})$ for some Hilbert space $(\mathscr{H}, [., .])$. Since
$$\|\langle x + \lambda y, x + \lambda y\rangle\| = \sup\Big\{[\langle x + \lambda y, x + \lambda y\rangle\xi, \xi]:\, \xi\in\mathscr{H},\|\xi\|=1\Big\},$$
there exists a sequence of unit vectors $\{\xi_n\}$ in $\mathscr{H}$ such that
\begin{align}\label{id.14}
\lim_{n\rightarrow\infty} [\langle x + \lambda y, x + \lambda y\rangle\xi_n, \xi_n] = \|\langle x + \lambda y, x + \lambda y\rangle\| = \|x + \lambda y\|^2 = (\|x\| + \|y\|)^2.
\end{align}
We have
\begin{align*}
[\langle x + \lambda y, x + \lambda y\rangle\xi_n, \xi_n ]&= [\langle x, x\rangle\xi_n, \xi_n] + 2\mbox{Re}[\langle x, \lambda y\rangle\xi_n, \xi_n] + [\langle y, y\rangle\xi_n, \xi_n]
\\& \leq \|x\|^2 + 2\mbox{Re}[\langle x, \lambda y\rangle\xi_n, \xi_n] + \|y\|^2
\\& \leq \|x\|^2 + 2\Big|[\langle x, y\rangle\xi_n, \xi_n]\Big| + \|y\|^2
\\& \leq \|x\|^2 + 2\|\langle x, y\rangle\| + \|y\|^2
\\& \leq \|x\|^2 + 2\|x\|\,\|y\| + \|y\|^2,
\end{align*}
whence
\begin{align}\label{id.14.5}
[\langle x + \lambda y, x + \lambda y\rangle\xi_n, \xi_n ] \leq \|x\|^2 + 2\mbox{Re}[\langle x, \lambda y\rangle\xi_n, \xi_n] + \|y\|^2 \leq \|x\|^2 + 2\|x\|\,\|y\| + \|y\|^2
\end{align}
and
\begin{align}\label{id.14.6}
[\langle x + \lambda y, x + \lambda y\rangle\xi_n, \xi_n ] \leq \|x\|^2 + 2\Big|[\langle x, y\rangle\xi_n, \xi_n]\Big| + \|y\|^2 \leq \|x\|^2 + 2\|x\|\,\|y\| + \|y\|^2.
\end{align}
From (\ref{id.14.5}) and (\ref{id.14.6}) we get
\begin{align}\label{id.15}
\frac{1}{2}\Big([\langle x + \lambda y, x + \lambda y\rangle\xi_n, \xi_n ] - \|x\|^2 - \|y\|^2\Big) \leq \mbox{Re}[\langle x, \lambda y\rangle\xi_n, \xi_n] \leq \|x\|\,\|y\|
\end{align}
and
\begin{align}\label{id.15.1}
\frac{1}{2}\Big([\langle x + \lambda y, x + \lambda y\rangle\xi_n, \xi_n ] - \|x\|^2 - \|y\|^2\Big) \leq \Big|[\langle x, y\rangle\xi_n, \xi_n]\Big| \leq \|x\|\,\|y\|.
\end{align}
According to (\ref{id.14}) we obtain
\begin{align}\label{id.15.11}
\lim_{n\rightarrow\infty} \frac{1}{2}\Big([\langle x + \lambda y, x + \lambda y\rangle\xi_n, \xi_n ] - \|x\|^2 - \|y\|^2\Big) = \|x\|\,\|y\|.
\end{align}
By (\ref{id.15}), (\ref{id.15.1}), (\ref{id.15.11}) and the Squeeze Theorem we conclude that $\lim_{n\rightarrow\infty} \mbox{Re}[\langle x, \lambda y\rangle\xi_n, \xi_n]$ and $\lim_{n\rightarrow\infty} \Big|[\langle x, y\rangle\xi_n, \xi_n]\Big|$ exist and
$$\lim_{n\rightarrow\infty} \mbox{Re}[\langle x, \lambda y\rangle\xi_n, \xi_n]= \|x\|\,\|y\| \qquad \mbox{and} \qquad \lim_{n\rightarrow\infty} \Big|[\langle x, y\rangle\xi_n, \xi_n]\Big| = \|x\|\,\|y\|.$$
(ii)$\Rightarrow$(i) Suppose that (ii) holds. We may assume that $x \neq 0$. By the Cauchy–-Schwarz inequality we have
\begin{align*}
\mbox{Re}[\langle x, \lambda y\rangle\xi_n, \xi_n] &\leq \Big|[\langle x, y\rangle\xi_n, \xi_n]\Big| \leq \|\langle x, y\rangle\xi_n\|
\\& = \sqrt{[{\langle x, y\rangle}^*\langle x, y\rangle\xi_n, \xi_n]} \qquad(\,\mbox{since}\,\, {\langle x, y\rangle}^*\langle x, y\rangle \leq \|\langle x, x\rangle\|\langle y, y\rangle)
\\&\leq \sqrt{[\|x\|^2\langle y, y\rangle\xi_n, \xi_n]} = \|x\| \sqrt{[\langle y, y\rangle\xi_n, \xi_n]}
\\&\leq \|x\| \sqrt{\|\langle y, y\rangle\xi_n\|}\leq \|x\|\,\|y\|,
\end{align*}
whence
\begin{align}\label{id.15.2}
\frac{(\mbox{Re}[\langle x, \lambda y\rangle\xi_n, \xi_n])^2}{\|x\|^2}\leq [\langle y, y\rangle\xi_n, \xi_n] \leq \|y\|^2.
\end{align}
Hence, by taking limits in (\ref{id.15.2}) and employing assumption (ii), we get
\begin{align}\label{id.16}
\lim_{n\rightarrow\infty} [\langle y, y\rangle\xi_n, \xi_n] = \|y\|^2.
\end{align}
By using a similar argument we conclude that
\begin{align}\label{id.17}
\lim_{n\rightarrow\infty} [\langle x, x\rangle\xi_n, \xi_n]=\|x\|^2.
\end{align}
By (ii), (\ref{id.16}) and (\ref{id.17}) we therefore reach
\begin{align*}
(\|x\| + \|y\|)^2 &= \lim_{n\rightarrow\infty} [\langle x, x\rangle\xi_n, \xi_n] + 2\lim_{n\rightarrow\infty} \mbox{Re}[\langle x, \lambda y\rangle\xi_n, \xi_n] + \lim_{n\rightarrow\infty} [\langle y, y\rangle\xi_n, \xi_n]
\\& = \lim_{n\rightarrow\infty} [\langle x + \lambda y, x + \lambda y\rangle\xi_n, \xi_n]
\\& \leq \|x + \lambda y\|^2 \leq (\|x\| + \|y\|)^2.
\end{align*}
Hence $\|x + \lambda y\| = \|x\| + \|y\|$. Thus $x\parallel y$.
\end{proof}
If $x, y$ are elements of a normed space $\mathscr{V}$ then, by the Hahn–-Banach theorem,  $x$ is orthogonal to $y$
in the Birkhoff--James sense if and only if there is a norm one linear functional $f$ on $\mathscr{V}$ such that $f(x) = \|x\|$ and $f(y) = 0$.
In the case of a Hilbert $\mathscr A$-module the Birkhoff--James
orthogonality can be described in the way which is more natural for Hilbert $C^*$-modules, i.e., $x\perp_{BJ} y$ if and only if there is a state $\varphi$ over $\mathscr A$ such that $\varphi(\langle x, x\rangle)=\|x\|^2$ and $\varphi(\langle x, y\rangle)=0$ (see \cite[Theorem 2.7]{A.R.2} and \cite[Theorem 4.4]{B.G}).
Now, we apply \cite[Theorem 2]{N.T} to obtain the following result.
\begin{theorem}\label{th.005}
Let $\mathscr{V}$ be a normed space. For $x, y\in \mathscr{V}$ there exists $\lambda\in\mathbb{T}$ such that the following statements are equivalent:
\begin{itemize}
\item[(i)] $x\parallel y$.
\item[(ii)] $x\perp_{BJ}(\|y\|x + \lambda \|x\|y)$.
\item[(ii)] $y\perp_{BJ}(\|x\|y + \overline{\lambda} \|y\|x)$.
\end{itemize}
\end{theorem}
\begin{corollary}\label{cr.17.9}
Let $x , y\in \mathscr{X}$. If $x\parallel y$, then there exists $\lambda\in\mathbb{T}$ such that
$$x\perp_{BJ}(\|y\|x\langle x, x\rangle + \lambda \|x\|x\langle x, y\rangle) \qquad \mbox{and} \qquad y\perp_{BJ}(\|x\|y\langle y, y\rangle + \lambda \|y\|y\langle x, y\rangle).$$
\end{corollary}
\begin{proof}
We may assume that $x\neq0$. By Theorem \ref{th.005}, there exists $\lambda\in\mathbb{T}$ such that
\begin{align}\label{id.17.9.1}
x\perp_{BJ}(\|y\|x + \lambda \|x\|y)
\end{align}
and
\begin{align}\label{id.17.9.2}
y\perp_{BJ}(\|x\|y + \overline{\lambda} \|y\|x).
\end{align}
By (\ref{id.17.9.1}) and \cite[Theorem 2.7]{A.R.2}, there is a state $\varphi$ over $\mathscr A$ such that
\begin{align}\label{id.17.91}
\varphi(\langle x, x\rangle) = \|x\|^2 \quad \mbox{and} \quad \varphi(\langle x, \|y\|x + \lambda \|x\|y\rangle) = 0.
\end{align}
Thus for every $\gamma\in\mathbb{C}$, by (\ref{id.17.91}), we obtain
\begin{align*}
\|\langle x, x\rangle\| = \|x\|^2 = \varphi\Big(\langle x, x\rangle + \gamma\langle x, \|y\|x + \lambda \|x\|y\rangle\Big) \leq \Big\|\langle x, x\rangle + \gamma\langle x, \|y\|x + \lambda \|x\|y\rangle\Big\|.
\end{align*}
Hence
\begin{align*}
\langle x, x\rangle \perp_{BJ}\langle x, \|y\|x + \lambda \|x\|y\rangle.
\end{align*}
So, there is a state $\psi$ over $\mathscr A$ such that
\begin{align*}
\psi({\langle x, x\rangle}^*\langle x, x\rangle)=\|\langle x, x\rangle\|^2 = \|x\|^4 \quad \mbox{and} \quad \psi({\langle x, x\rangle}^*\langle x, \|y\|x + \lambda \|x\|y\rangle) = 0.
\end{align*}
Since $\|x\|^4 = \psi({\langle x, x\rangle}^*\langle x, x\rangle) \leq \psi(\|x\|^2 \langle x, x\rangle) \leq \|x\|^4$, we get
\begin{align*}
\psi(\langle x, x\rangle) = \|x\|^2\,\, \mbox{and}\,\, \psi\Big(\langle x, \|y\|x\langle x, x\rangle + \lambda \|x\|x\langle x, y\rangle\rangle\Big) = \psi({\langle x, x\rangle}^*\langle x, \|y\|x + \lambda \|x\|y\rangle) = 0.
\end{align*}
Thus we conclude that $$x\perp_{BJ}(\|y\|x\langle x, x\rangle + \lambda \|x\|x\langle x, y\rangle).$$
By (\ref{id.17.9.2}) and a similar computation as above we get
$$y\perp_{BJ}(\|x\|y\langle y, y\rangle + \lambda \|y\|y\langle x, y\rangle).$$
\end{proof}
We have the following characterization of the norm-parallelism for certain elements of a Hilbert $C^{*}$-module.
\begin{theorem}\label{th.18}
Let $x , y\in \mathscr{X}$. If $\langle x, y\rangle$ is normal, then the following statements are equivalent:
\begin{itemize}
\item[(i)] $x\parallel y$.
\item[(ii)] $\|\langle x, y\rangle\| = \|x\|\,\|y\|$.
\end{itemize}
\end{theorem}
\begin{proof}
(i)$\Rightarrow$(ii) This implication follows from the equivalence $(i) \Leftrightarrow (iii)$ of Lemma \ref{lemma.280}.\\
(ii)$\Rightarrow$(i) Suppose that $\|\langle x, y\rangle\| = \|x\|\,\|y\|$ holds. By \cite[Theorem 3.3.6]{mor}, for normal element $\langle x, y\rangle \in{\mathscr A}$, there exists a state $\varphi$ over ${\mathscr A}$ such that $|\varphi(\langle x, y\rangle)| = \|\langle x, y\rangle\| = \|x\|\,\|y\|$. Thus, by the equivalence $(i) \Leftrightarrow (ii)$ of Lemma \ref {lemma.280}, we conclude that $x\parallel y$.
\end{proof}
The following example shows that the normality condition cannot be
omitted in the implication (ii)$\Rightarrow$(i) of Theorem \ref{th.18}.
\begin{example}\label{ex.19}
Let $A =\begin{bmatrix}
1 & 1 \\
0 & 0
\end{bmatrix}\in \mathbb{B}(\mathbb{C}^2)=\mathbb{M}_2(\mathbb{C})$. Then $\langle A, I\rangle = A^*$ is not normal. We have
$$\|\langle A, I\rangle\| = \|A\|\,\|I\| = \sqrt{2}.$$
Since $r(A) = 1 \neq \sqrt{2} = \|A\|$, by the equivalence $(i) \Leftrightarrow (iv)$ of Lemma \ref{lemma.280}, we reach $A\nparallel I$.\\
Also, let $B =\begin{bmatrix}
2 & 5 \\
5 & 0
\end{bmatrix}$, $C =\begin{bmatrix}
1 & 0 \\
0 & -1
\end{bmatrix}$. Since $B$ and $C$ are self-adjoint, by \cite[Theorem 2.1.1]{mor}, we have $r(B) = \|B\|$ and $r(C) = \|C\|$. Thus, by the equivalence $(i) \Leftrightarrow (iv)$ of Lemma \ref{lemma.280}, we get $B\parallel I$ and $I\parallel C$. A simple computation shows that
$$\|\langle B, C\rangle\| = \Big\|\begin{bmatrix}
2 & 5 \\
-5 & 0
\end{bmatrix}\Big\| = \sqrt{26} + 1 \neq 5 = r(\langle B, C\rangle).$$
So, by the equivalence $(i) \Leftrightarrow (iv)$ of Lemma \ref{lemma.280}, we observe that $B\nparallel C$. Hence the norm-parallelism is not transitive.
\end{example}
\begin{proposition}\label{th.18.1}
Let ${\mathscr A}$ be a unital $C^*$-algebra with the identity $e$ and $x , y, z\in \mathscr{X}$. If $\langle z, z\rangle$ is invertible and $\langle x, y\rangle\in\mathbb{C}\cdot e$, then the following statements are equivalent:
\begin{itemize}
\item[(i)] $x\parallel y$.
\item[(ii)] $\theta_{z,x}\parallel \theta_{z,y}$.
\end{itemize}
\end{proposition}
\begin{proof}
First notice that, by replacing $z$ with ${\langle z, z\rangle}^\frac{-1}{2}z$, we assume that $\langle z, z\rangle = e$. Further, since $\langle x, y\rangle\in\mathbb{C}\cdot e$ we get ${\theta}^2_{x,y} = \theta_{x\langle y, x\rangle,y} = \langle y, x\rangle\theta_{x,y},$ from which we obtain
${\theta}^n_{x,y} = {\langle y, x\rangle}^{n - 1}\theta_{x,y} \,\, (n\in\mathbb{N})$. It follows from $\lim_{n\rightarrow\infty} \|{\theta}^n_{x,y}\|^{\frac{1}{n}} = \lim_{n\rightarrow\infty} |{\langle y, x\rangle}|^{\frac{n - 1}{n}} \|\theta_{x,y}\|^{\frac{1}{n}} = |{\langle y, x\rangle}|$ that
\begin{align}\label{id.18.2}
r(\theta_{x,y}) = |{\langle y, x\rangle}|.
\end{align}
We have
\begin{align*}
x\parallel y & \Longleftrightarrow |{\langle y, x\rangle}| = \|y\|\,\|x\| \qquad\qquad \qquad \qquad(\mbox{by Theorem} \,\,\ref{th.18})
\\& \Longleftrightarrow r(\theta_{x,y}) = \|x\|\,\|y\|  \qquad\qquad \qquad\qquad (\mbox{by} \,\,\ref{id.18.2})
\\& \Longleftrightarrow r(\theta_{x\langle z, z\rangle,y}) = \|\theta_{x\langle z, z\rangle,y}\| = \|x\|\,\|y\|\qquad\quad \qquad \qquad \qquad (\mbox{since} \,\,\langle z, z\rangle = e)
\\& \Longleftrightarrow r(\theta_{x,z}\theta_{z,y}) = \|\theta_{x,z}\theta_{z,y}\| = \|z\|\,\|x\|\,\|z\|\,\|y\| \qquad  \qquad (\mbox{since} \,\,\|z\| = 1)
\\& \Longleftrightarrow r({\theta}^\ast_{z,x}\theta_{z,y}) = \|{\theta}^\ast_{z,x}\theta_{z,y}\| = \|\theta_{z,x}\|\,\|\theta_{z,y}\|
\\& \Longleftrightarrow \theta_{z,x}\parallel \theta_{z,y}.\qquad  \qquad \qquad (\mbox{by the equivalence}\,\, (i) \Leftrightarrow (iv) \,\mbox{of Lemma} \,\,\ref{lemma.280})
\end{align*}
\end{proof}
Recall that the numerical range of an arbitrary element $a\in{\mathscr A}$ is defined by $W_0(a) = \{\varphi(a):\,\varphi \,\,\mbox{is a state on}\,\, {\mathscr A}\}$. $W_0(a)$ is convex, closed and contains the spectrum $\sigma(a)$. Furthermore, if $a$ is normal, then $W_0(a)$ is the convex hull of the spectrum of $a$ (see \cite[Theorems 1, 8]{S.W}). By $\mbox{Re}(a) = \frac{a + a^*}{2}$ we shall denote the real part of $a$.
Before the next theorem we shall first recall the following result.
\begin{lemma} \cite[Corollary 4.2]{M.Z}\label{le.21}
Let $x , y\in \mathscr{X}\smallsetminus\{0\}$.
\begin{itemize}
\item[(i)] If $x\parallel y$, then there exist a state $\varphi$ over $\mathscr A$ and $\lambda\in\mathbb{T}$ such that $$\frac{\|y\|}{\|x\|}\varphi(|x|^2)+\frac{\|x\|}{\|y\|}\varphi(|y|^2)=2\lambda\varphi({\langle x, y\rangle}).$$
\item[(ii)] Let ${\mathscr A}$ has an identity $e$. If either $|x|^2=e$ or $|y|^2=e$ and there exist a state $\varphi$ over $\mathscr A$ and $\lambda\in\mathbb{T}$ such that $\frac{\|y\|}{\|x\|}\varphi(|x|^2)+\frac{\|x\|}{\|y\|}\varphi(|y|^2)=2\lambda\varphi({\langle x, y\rangle})$, then $x\parallel y$.
\end{itemize}
\end{lemma}
\begin{theorem}\label{th.22}
Let ${\mathscr A}$ be a unital $C^*$-algebra with the identity $e$ and $x , y\in \mathscr{X}$ such that $|y|^2 = e$. The following statements are equivalent:
\begin{itemize}
\item[(i)] $x\parallel y$.
\item[(ii)] $|x|^2-2\|x\|\mbox{Re}(\langle x, \lambda y\rangle) + \|x\|^2e $ is non-invertible for some $\lambda\in\mathbb{T}$.
\end{itemize}
\end{theorem}
\begin{proof}
(i)$\Rightarrow$(ii) Let $x\parallel y$. We may assume that $x\neq 0$. Then, by Lemma \ref{le.21} (i), there exist a state $\varphi$ over $\mathscr A$ and $\lambda\in\mathbb{T}$ such that
\begin{align*}
\frac{1}{\|x\|}\varphi(|x|^2)+\|x\|\varphi(e) = \frac{\|y\|}{\|x\|}\varphi(|x|^2)+\frac{\|x\|}{\|y\|}\varphi(|y|^2)=2\lambda\varphi({\langle x, y\rangle}).
\end{align*}
Hence
$$\varphi\Big(|x|^2-2\|x\|\mbox{Re}(\langle x, \lambda y\rangle) + \|x\|^2e \Big) = 0.$$
So, $W_0\Big(|x|^2-2\|x\|\mbox{Re}(\langle x, \lambda y\rangle) + \|x\|^2e\Big)$ is the line segment connecting $0$ and $\Big\||x|^2-2\|x\|\mbox{Re}(\langle x, \lambda y\rangle) + \|x\|^2e\Big\|$.
Furthermore, by \cite[Theorems 1, 8]{S.W} $W_0\Big(|x|^2-2\|x\|\mbox{Re}(\langle x, \lambda y\rangle) + \|x\|^2e\Big)$ is the convex hull of the spectrum of $|x|^2-2\|x\|\mbox{Re}(\langle x, \lambda y\rangle) + \|x\|^2e$. Hence we get $0\in \sigma\Big(|x|^2-2\|x\|\mbox{Re}(\langle x, \lambda y\rangle) + \|x\|^2e \Big)$. So, $|x|^2-2\|x\|\mbox{Re}(\langle x, \lambda y\rangle) + \|x\|^2e $ is non-invertible.\\
(ii)$\Rightarrow$(i) Suppose that (ii) holds. We have
$$0\in \sigma\Big(|x|^2-2\|x\|\mbox{Re}(\langle x, \lambda y\rangle) + \|x\|^2e \Big) \subseteq W_0\Big(|x|^2-2\|x\|\mbox{Re}(\langle x, \lambda y\rangle) + \|x\|^2e \Big).$$
Hence there exists a state $\varphi$ over $\mathscr A$ such that
\begin{align}\label{id.25}
\varphi\Big(|x|^2-2\|x\|\mbox{Re}(\langle x, \lambda y\rangle) + \|x\|^2e \Big) = 0.
\end{align}
Thus
\begin{align}\label{id.26}
\frac{\|y\|}{\|x\|}\varphi(|x|^2)+\frac{\|x\|}{\|y\|}\varphi(|y|^2)= \frac{1}{\|x\|}\varphi(|x|^2)+\|x\|\varphi(e) = 2\varphi\Big(\mbox{Re}(\langle x, \lambda y\rangle)\Big).
\end{align}
Therefore
\begin{align*}
0\leq\Big(\sqrt{\varphi(|x|^2)}-\|x\|\Big)^2& = \varphi(|x|^2)-2\|x\|\sqrt{\varphi(|x|^2)} + \|x\|^2
\\& = \varphi(|x|^2)-2\|x\|\sqrt{\varphi(\langle x, x\rangle)\varphi(\langle \lambda y, \lambda y\rangle)} + \|x\|^2
\\& \hspace{5cm}(\mbox{since}\,\,\varphi(\langle \lambda y, \lambda y\rangle)= \varphi(e)=1)
\\&\leq \varphi(|x|^2)-2\|x\|\Big|\varphi(\langle x, \lambda y\rangle)\Big| + \|x\|^2
\\&\hspace{4cm}(\mbox{by the Cauchy--Schwarz inequality})
\\& \leq \varphi(|x|^2)-2\|x\|\mbox{Re}\Big(\varphi(\langle x, \lambda y\rangle)\Big) + \|x\|^2
\\&= \varphi(|x|^2)-2\|x\|\varphi\Big(\mbox{Re}(\langle x, \lambda y\rangle)\Big) + \|x\|^2\varphi(e)
\\&= \varphi\Big(|x|^2-2\|x\|\mbox{Re}(\langle x, \lambda y\rangle) + \|x\|^2e \Big) = 0  \qquad\qquad(\mbox{by}\,\,(\ref{id.25})).
\end{align*}
We conclude that $\varphi\Big(\mbox{Re}(\langle x, \lambda y\rangle)\Big) = \Big|\varphi(\langle x, \lambda y\rangle)\Big| = \varphi(\langle x, \lambda y\rangle)$, since $\varphi\Big(\mbox{Re}(\langle x, \lambda y\rangle)\Big)\geq 0$. It follows from (\ref{id.26}) that
$$\frac{\|y\|}{\|x\|}\varphi(|x|^2)+\frac{\|x\|}{\|y\|}\varphi(|y|^2) = 2\varphi(\langle x, \lambda y\rangle) = 2\lambda \varphi(\langle x, y\rangle),$$
whence, by Lemma \ref{le.21} (ii), $x\parallel y$.
\end{proof}
\begin{corollary}\label{co.27}
Let ${\mathscr A}$ be a unital commutative $C^*$-algebra with the identity $e$ and $a , b\in \mathscr{X}$ such that $|b|^2 = e$. The following statements are equivalent:
\begin{itemize}
\item[(i)] $a\parallel b$.
\item[(ii)] $a-\lambda\|a\|b$ is non-invertible for some $\lambda\in\mathbb{T}$.
\end{itemize}
\end{corollary}
\begin{proof}
Since ${\mathscr A}$ is commutative, $|a|^2-2\|a\|\mbox{Re}(\langle a, \lambda b\rangle) + \|a\|^2e = (a-\lambda\|a\|b)^*(a- \lambda\|a\|b)$ is non-invertible if and only if $a-\lambda\|a\|b$ is non-invertible. Therefore, the statement follows from Theorem \ref{th.22}.
\end{proof}

As a consequence of Theorem \ref{th.05}, we have the following characterization of the norm-parallelism for elements of $\mathbb{B}(\mathscr{H})$.
\begin{corollary}\label{cr.055}
For $T, S\in\mathbb{B}(\mathscr{H})$ the following statements are equivalent:
\begin{itemize}
\item[(i)] $T\parallel S$.
\item[(ii)] There exists a sequence of unit vectors $\{\xi_n\}$ in $\mathscr{H}$ such that $$\lim_{n\rightarrow\infty} |[T\xi_n, S\xi_n]| = \|T\|\,\|S\|.$$
\end{itemize}
In addition, if $\{\xi_n\}$ is a sequence of unit vectors in $\mathscr{H}$ satisfying (ii), then it also satisfies
$$\lim_{n\rightarrow\infty} \|T\xi_n\| = \|T\| \quad \mbox{and} \quad \lim_{n\rightarrow\infty} \|S\xi_n\| = \|S\|.$$
\end{corollary}
\begin{proof}
(i)$\Rightarrow$(ii) This implication follows immediately from Theorem \ref{th.05}.\\
(ii)$\Rightarrow$(i)
Suppose that there exist a sequence of unit vectors $\{\xi_n\}$ in $\mathscr{H}$ and $\lambda\in\mathbb{T}$ such that $\lim_{n\rightarrow\infty} [T\xi_n, S\xi_n] = \lambda \|T\|\,\|S\|$. It follows from
$$\|T\|\,\|S\| = \lim_{n\rightarrow\infty} |[T\xi_n, S\xi_n]|\leq \lim_{n\rightarrow\infty} \|T\xi_n\|\,\|S\|\leq\|T\|\,\|S\|,$$
that $\lim_{n\rightarrow\infty} \|T\xi_n\|=\|T\|$ and by using a similar argument, $\lim_{n\rightarrow\infty} \|S\xi_n\|=\|S\|$.
Hence
$$\lim_{n\rightarrow\infty} \mbox{Re}[T\xi_n, \lambda S\xi_n] = \lim_{n\rightarrow\infty} [T\xi_n, \lambda S\xi_n] = \|T\|\,\|S\|.$$
Thus, by Theorem \ref{th.05}, we get
$T\parallel S$.
\end{proof}
Next we obtain some characterizations of the norm-parallelism for elements of $\mathbb{B}(\mathscr{H})$ defined on a finite dimensional Hilbert space.
\begin{theorem}\label{th.17.5}
Let $\mathscr{H}$ be a finite dimensional Hilbert space and $T, S\in\mathbb{B}(\mathscr{H})$. The following statements are equivalent:
\begin{itemize}
\item[(i)] $T\parallel S$.
\item[(ii)] There exists a unit vector $\xi\in\mathscr{H}$ such that $|[T\xi, S\xi]| = \|T\|\,\|S\|$.
\item[(iii)] There exists a unit vector $\xi\in\mathscr{H}$ such that $\|T\xi\| = \|T\|$, $\|S\xi\| = \|S\|$ and $T\xi\parallel S\xi$.
\end{itemize}
\end{theorem}
\begin{proof}
(i)$\Rightarrow$(ii) Let $T\parallel S$. By Theorem \ref{th.05}, there exist a sequence of unit vectors $\{\xi_n\}$ in $\mathscr{H}$ and $\lambda\in\mathbb{T}$ such that
$$\lim_{n\rightarrow\infty} \mbox{Re}[\langle T, \lambda S\rangle\xi_n, \xi_n]= \|T\|\,\|S\|.$$
Since $\{\xi_n\}$ is a bounded sequence, it has a convergent subsequence converging to a unit vector $\xi$. Thus we obtain
\begin{align}\label{id.17.51}
\mbox{Re}[\langle T, \lambda S\rangle\xi, \xi]= \|T\|\,\|S\|.
\end{align}
Now, let $\lambda = \sin\theta + i\cos\theta$ for some $\theta\in [0, \pi]$. Then from (\ref{id.17.51}) we get
\begin{align*}
\|T\|\,\|S\| &= \mbox{Re}[\langle T, \lambda S\rangle\xi, \xi] = \mbox{Re}[\lambda T^*S\xi, \xi]
\\& = \mbox{Re}\Big((\sin\theta + i\cos\theta)[S\xi, T\xi]\Big)
\\& = \mbox{Re}\Big((\sin\theta + i\cos\theta)(\mbox{Re}[S\xi, T\xi] + i \,\mbox{Im}[S\xi, T\xi])\Big)
\\& = \mbox{Re}\Big(\mbox{Re}[S\xi, T\xi]\sin\theta - \mbox{Im}[S\xi, T\xi]\cos\theta + i\,(\mbox{Im}[S\xi, T\xi]\sin\theta + \mbox{Re}[S\xi, T\xi]\cos\theta)\Big)
\\& = \mbox{Re}[S\xi, T\xi]\sin\theta - \mbox{Im}[S\xi, T\xi]\cos\theta.
\end{align*}
Thus $\|T\|\,\|S\| = \mbox{Re}[S\xi, T\xi]\sin\theta - \mbox{Im}[S\xi, T\xi]\cos\theta$.
We have
\begin{align*}
\|T\|\,\|S\| &= \Big|\mbox{Re}[S\xi, T\xi]\sin\theta - \mbox{Im}[S\xi, T\xi]\cos\theta\Big|
\\& \leq \sqrt{\mbox{Re}^2[S\xi, T\xi] + \mbox{Im}^2[S\xi, T\xi]} = |[S\xi, T\xi]|\leq \|S\xi\|\,\|T\xi\|\leq \|T\|\,\|S\|.
\end{align*}
Thus $|[T\xi, S\xi]| = |[S\xi, T\xi]| = \|T\|\,\|S\|$.\\
(ii)$\Rightarrow$(iii) Suppose that (ii) holds. Therefore we have
$$\|T\|\,\|S\| = |[T\xi, S\xi]|\leq \|T\xi\|\,\|S\xi\|\leq \|T\|\,\|S\xi\|\leq \|T\|\,\|S\|.$$
This gives
$$\|T\xi\| = \|T\| \quad \mbox{and} \quad \|S\xi\| = \|S\|.$$
Hence we have
$$|[T\xi, S\xi]| = \|T\|\,\|S\| = \|T\xi\|\,\|S\xi\|,$$
i.e., $T\xi$ and $S\xi$ are linearly dependent. Thus $T\xi\parallel S\xi$.\\
(iii)$\Rightarrow$(i) Let there exist a unit vector $\xi\in\mathscr{H}$ such that $\|T\xi\| = \|T\|$, $\|S\xi\| = \|S\|$ and $T\xi\parallel S\xi$. Therefore there exists $\lambda\in\mathbb{T}$ such that
$$\|T\xi + \lambda S\xi\| = \|T\xi\| + \|S\xi\|.$$
So, we get
$$\|T\|+\|S\| = \|T\xi\| + \|S\xi\| = \|T\xi + \lambda S\xi\| \leq \|T + \lambda S\| \leq \|T\| + \|S\|.$$
Thus $\|T + \lambda S\| = \|T\| + \|S\|$, so $T\parallel S$.
\end{proof}
\begin{corollary}\label{cr.17.6}
Let $\mathscr{H}$ be finite dimensional, $T\in\mathbb{B}(\mathscr{H})$ be a nonzero positive operator and $S\in\mathbb{B}(\mathscr{H})$ be an arbitrary operator. Then the following statements are equivalent:
\begin{itemize}
\item[(i)] $T\parallel S$.
\item[(ii)] There exists a unit vector $\xi\in\mathscr{H}$ such that $T\xi= \|T\|\xi$ and $|[S\xi, \xi]| = \|S\|$.
\end{itemize}
\end{corollary}
\begin{proof}
(i)$\Rightarrow$(ii) Let $T\parallel S$. By the equivalence $(i) \Leftrightarrow (ii)$ of Theorem \ref{th.17.5}, there exists a unit vector $\xi\in\mathscr{H}$ such that $|[T\xi, S\xi]| = \|T\|\,\|S\|$. So, we get $\|T\xi\|= \|T\|$. Since $T$ is positive, we obtain $T\xi = \|T\|\xi$. Therefore,
$$\|T\|\,|[\xi, S\xi]| = |[\|T\|\xi, S\xi]| = |[T\xi, S\xi]| = \|T\|\,\|S\|,$$
or equivalently, $|[S\xi, \xi]| = \|S\|$.\\
The implication (ii)$\Rightarrow$(i) follows by the same argument.
\end{proof}
\begin{corollary}\label{cr.17.7}
Let $\mathscr{H}$ be finite dimensional and $T, S\in\mathbb{B}(\mathscr{H})$. The following statements are equivalent:
\begin{itemize}
\item[(i)] $T\parallel S$.
\item[(ii)] There exists a unit vector $\xi\in\mathscr{H}$ such that $T^*T\xi = \|T\|^2\xi$ and $|[T\xi, S\xi]| = \|T\|\,\|S\|$.
\end{itemize}
\end{corollary}
\begin{proof}
By the equivalence $(i) \Leftrightarrow (iii)$ of Lemma \ref{lemma.280}, $T\parallel S$ if and only if $T^*T\parallel T^*S$ and $\|T^*S\| = \|T\|\,\|S\|$. Since $T^*T$ is positive, so the statement follows from Corollary \ref{cr.17.6}.
\end{proof}
As an application of Corollary \ref{cr.17.6} we have the following result.
\begin{corollary}\label{cr.17.8}
Let $T_i\in\mathbb{B}(\mathscr{H}_i),\,(1\leq i\leq n)$ be nonzero positive operators on finite dimensional Hilbert spaces. Then for every $T_i, S_i\in\mathbb{B}(\mathscr{H}_i),\,(1\leq i\leq n)$ the following statements are equivalent:
\begin{itemize}
\item[(i)] ${\rm diag} (T_1, \cdots, T_n) \| {\rm diag} (S_1, \cdots, S_n)$.
\item[(ii)] There exists a unit vector $(\xi_1, \cdots, \xi_n)\in\mathscr{H}_1\oplus\cdots \oplus\mathscr{H}_n$ such that
$$\Big(\sum_{i = 1}^n \|T_i\xi_i\|^2\Big)^\frac{1}{2} = \max\{\|T_i\|; \quad 1\leq i\leq n\}\quad \mbox{and} \quad |\sum_{i = 1}^n[S_i\xi_i, \xi_i]| = \max\{\|S_i\|; \quad 1\leq i\leq n\}.$$
\end{itemize}
\end{corollary}
If $T\in\mathbb{B}(\mathscr{H})$, then by Corollary \ref{cr.055} for
any $S\in\mathbb{B}(\mathscr{H})$, $T\parallel S$ if and only if there exists a sequence of unit vectors $\{\xi_n\}$ in $\mathscr{H}$ such that $\lim_{n\rightarrow\infty} |[T\xi_n, S\xi_n]| = \|T\|\,\|S\|.$
It is easy to see that if there exists a unit vector $\xi\in\mathscr{H}$ such that $|[T\xi, S\xi]| = \|T\|\,\|S\|$, then $T\parallel S$. The question is under which conditions the converse is true. When the
Hilbert space is finite dimensional, it follows from the equivalence $(i) \Leftrightarrow (ii)$ of Theorem \ref{th.17.5}, there exists a unit vector $\xi\in\mathscr{H}$ such that $|[T\xi, S\xi]| = \|T\|\,\|S\|$.
The following example shows that the condition finite dimensional in the implication $(i)\Rightarrow (ii)$ of Theorem \ref{th.17.5} is essential.
\begin{example}\label{ex.17.97}
Consider the shift operator $T\,:\ell^2\longrightarrow \ell^2$ defined by
$$T(\xi_1, \xi_2, \xi_3, \cdots) = (0, \xi_1, \xi_2, \xi_3, \cdots).$$
One can easily observe that $r(T) = 1 = \|T\|$. By the equivalence $(i) \Leftrightarrow (iv)$ of Lemma \ref{lemma.280}, we get $T \parallel I$. But there is no unit vector $\xi\in\ell^2$ such that $|[T\xi, I\xi]| = \|T\|\,\|I\|.$ Indeed, if there exists a unit vector $\xi\in\ell^2$ such that $|[T\xi, I\xi]| = \|T\|\,\|I\|$, then $$|0\overline{\xi_1} + \xi_1\overline{\xi_2} + \xi_2\overline{\xi_3} + \cdots| = 1 = (|0|^2 + |\xi_1|^2 + |\xi_2|^2 + \cdots)^\frac{1}{2}(|\xi_1|^2 + |\xi_2|^2 + |\xi_3|^2 + \cdots)^\frac{1}{2}.$$
It follows from the equality case in the Cauchy-Schwarz inequality, there exists $\gamma \in\mathbb{C}$ such that $(0, \xi_1, \xi_2, \xi_3, \cdots) = \gamma (\xi_1, \xi_2, \xi_3, \cdots)$. Thus $\xi_n = 0$ for all $n\in \mathbb{N}$. But this contradicts the fact that $\xi$ is a unit vector.
\end{example}
We now settle the problem for any infinite dimensional Hilbert space. We let $S_{\mathscr{H}} = \{\xi\in\mathscr{H}:\, \|\xi\| = 1\}$
and $M_T = \{\xi\in S_{\mathscr{H}}:\, \|T\xi\| = \|T\|\}$ be the unit sphere of $\mathscr{H}$ and the set of all unit vectors in $S_{\mathscr{H}}$ at which $T$ attains norm, respectively.
The proof of Theorem \ref{th.17.98} is a modification of one given by Paul et al. \cite[Theorem 3.1]{P.S.G}.
\begin{theorem}\label{th.17.98}
Let $T\in\mathbb{B}(\mathscr{H})$. If $S_{\mathscr{H}_0} = M_T$, where $\mathscr{H}_0$ is a finite dimensional
subspace of $\mathscr{H}$ and $\|T\|_{{\mathscr{H}_0}^\perp} = \sup\{\|T\zeta\| :\, \zeta\in{\mathscr{H}_0}^\perp,\, \|\zeta\| = 1\}< \|T\|$, then for any $S\in\mathbb{B}(\mathscr{H})$ the following statements are equivalent:
\begin{itemize}
\item[(i)] $T\parallel S$.
\item[(ii)] There exists a unit vector $\xi\in \mathscr{H}_0$ such that $|[T\xi, S\xi]| = \|T\|\,\|S\|$.
\end{itemize}
\end{theorem}
\begin{proof}
Obviously, (ii)$\Rightarrow$(i).

Suppose (i) holds. By Corollary \ref{cr.055}, there exists a sequence of unit vectors $\{\zeta_n\}$ in $\mathscr{H}$ such that
\begin{align}\label{id.17.981}
\lim_{n\rightarrow\infty} |[T\zeta_n, S\zeta_n]| = \|T\|\,\|S\|,\quad \lim_{n\rightarrow\infty} \|T\zeta_n\| = \|T\| \quad \mbox{and}\quad \lim_{n\rightarrow\infty} \|S\zeta_n\| = \|S\|.
\end{align}
For each $n\in\mathbb{N}$ we have
\begin{align}\label{id.17.981.1}
\zeta_n = \xi_n + \eta_n,
\end{align}
where $\xi_n \in \mathscr{H}_0$ and $\eta_n \in {\mathscr{H}_0}^\perp$.

Since $\mathscr{H}_0$ is a finite dimensional subspace and $\|\xi_n\| \leq 1$, so $\{\xi_n\}$ has a convergent subsequence
converging to some element of $\mathscr{H}_0$. Without loss of generality we assume that $\lim_{n\rightarrow\infty} \xi_n = \xi$. Since $S_{\mathscr{H}_0} = M_T$, so
\begin{align}\label{id.17.982}
\lim_{n\rightarrow\infty} \|T\xi_n\| = \|T\xi\| = \|T\|\,\|\xi\|
\end{align}
and
\begin{align}\label{id.17.983}
\lim_{n\rightarrow\infty} \|\eta_n\|^2 = \lim_{n\rightarrow\infty}( \|\zeta_n\|^2 - \|\xi_n\|^2 ) = 1 - \|\xi\|^2.
\end{align}
Now for each non-zero element $\xi_n \in \mathscr{H}_0$, by hypothesis $\frac{\xi_n}{\|\xi_n\|}\in S_{\mathscr{H}_0} = M_T $ and so $\|T\xi_n\| = \|T\|\,\|\xi_n\|$. Thus
\begin{align*}
\|T^*T\xi_n\|\,\|\xi_n\| \leq \|T^*T\|\,\|\xi_n\|^2 = \|T\|^2\,\|\xi_n\|^2 = \|T\xi_n\|^2 = [T^*T\xi_n, \xi_n] \leq \|T^*T\xi_n\|\,\|\xi_n\|.
\end{align*}
Hence $[T^*T\xi_n, \xi_n] = \|T^*T\xi_n\|\,\|\xi_n\|$. By the equality case of Cauchy--Schwarz inequality $T^*T\xi_n = \lambda_n\xi_n$ for some $\lambda_n\in\mathbb{C}$ and therefore
\begin{align}\label{id.17.985}
[T^*T\xi_n, \eta_n] = [T^*T\eta_n, \xi_n] = 0.
\end{align}
By (\ref{id.17.981.1}) and (\ref{id.17.985}) we have
\begin{align*}
\|T\zeta_n\|^2 & = [T^*T\zeta_n, \zeta_n]
\\& = [T^*T(\xi_n + \eta_n), (\xi_n + \eta_n)]
\\& = [T^*T\xi_n, \xi_n] + [T^*T\xi_n, \eta_n] + [T^*T\eta_n, \xi_n] + [T^*T\eta_n, \eta_n]
\\& = \|T\xi_n\|^2 + \|T\eta_n\|^2,
\end{align*}
or equivalently,
\begin{align}\label{id.17.985.1}
\|T\eta_n\|^2 = \|T\zeta_n\|^2 - \|T\xi_n\|^2.
\end{align}
By (\ref{id.17.981}), (\ref{id.17.982}), (\ref{id.17.983}) and (\ref{id.17.985.1}) we reach
\begin{align}\label{id.17.984}
\lim_{n\rightarrow\infty} \|T\eta_n\|^2 = \|T\|^2(1 - \|\xi\|^2) = \|T\|^2\,\lim_{n\rightarrow\infty} \|\eta_n\|^2.
\end{align}
By hypothesis $\|T\|_{{\mathscr{H}_0}^\perp} < \|T\|$ and so by (\ref{id.17.984}) there does
not exist any non-zero subsequence of $\{\|\eta_n\|\}$. (Indeed, if there exists a non-zero subsequence of $\{\|\eta_n\|\}$, then (\ref{id.17.984}) implies $\|T\|_{{\mathscr{H}_0}^\perp} = \|T\|$.)
So we conclude $\eta_n = 0$ for all $n\in\mathbb{N}$. Then (\ref{id.17.981}), (\ref{id.17.983}) imply
$\|\xi\| = 1$ and $$|[T\xi, S\xi]| = \lim_{n\rightarrow\infty} |[T\xi_n, S\xi_n]| = \lim_{n\rightarrow\infty} |[T\zeta_n, S\zeta_n]| = \|T\|\,\|S\|.$$
\end{proof}

In the following propositions we use some ideas of \cite{B.S}.
\begin{proposition}\label{pr.17.81}
Let $\mathscr{H}$ be a finite dimensional Hilbert space, $T, S\in\mathcal{C}_p$ and $T = U|T|, S = V|S|$ be their polar decompositions. If $1<p<\infty$, then the following statements are equivalent:
\begin{itemize}
\item[(i)] $T\parallel S$ in the Schatten $p$-norm.
\item[(ii)] $\|T\|_p\,\Big|{\rm tr}(|T|^{p-1}U^*S)\Big| = \|S\|_p\,{\rm tr}(|T|^p)$.
\item[(iii)] $\|S\|_p\,\Big|{\rm tr}(|S|^{p-1}V^*T)\Big| = \|T\|_p\,{\rm tr}(|S|^p)$.
\end{itemize}

The same is true for $p=1$ if $T, S\in\mathcal{C}_1$ are invertible.
\end{proposition}
\begin{proof}
(i)$\Rightarrow$(ii) Let (i) hold. By Theorem \ref{th.005}, there exists $\lambda\in\mathbb{T}$ such that
$T\perp_{BJ}(\|S\|_{p}T + \lambda \|T\|_{p}S)$. By \cite[Theorem 2.1]{B.S}, we get
$${\rm tr}[|T|^{p - 1}U^*(\|S\|_{p}T + \lambda \|T\|_{p}S)] = 0,$$
or equivalently,
$$\|T\|_{p}\,\Big|{\rm tr}(|T|^{p-1}U^*S)\Big| = \|S\|_{p}\,{\rm tr}(|T|^p).$$
(ii)$\Rightarrow$(i) Let (ii) hold. There exists $\lambda\in\mathbb{T}$ such that
$${\rm tr}[|T|^{p - 1}U^*(\|S\|_{p}T + \lambda \|T\|_{p}S)] = 0.$$
Hence
$${\rm tr}|T|^{p} = {\rm tr}[|T|^{p - 1}( |T| + \xi U^*(\|S\|_{p}T + \lambda \|T\|_{p}S))]$$
for all $\xi\in\mathbb{C}$. From the above and the H\"{o}lder inequality, we get
\begin{align*}
{\rm tr}(|T|^{p}) &\leq \Big\||T|^{p - 1}\Big\|_{\frac{1}{{1 - \frac{1}{p}}}}\,\Big\||T| + \xi U^*(\|S\|_{p}T + \lambda \|T\|_{p}S)\Big\|_{p}
\\& = [{\rm tr}(|T|^{p})]^{1 - \frac{1}{p}}\,\Big\|T + \xi (\|S\|_{p}T + \lambda \|T\|_{p}S)\Big\|_{p}.
\end{align*}
It follows that
$$\|T\|_{p} \leq \Big\|T + \xi (\|S\|_{p}T + \lambda \|T\|_{p}S)\Big\|_{p} \qquad (\xi\in\mathbb{C}),$$
or equivalently,
$$T\perp_{BJ}(\|S\|_{p}T + \lambda \|T\|_{p}S).$$
Thus, by Theorem \ref{th.005}, $T\parallel S$ in the Schatten $p$-norm.

(i)$\Longleftrightarrow$ (iii) It follows from the equivalence $(i) \Leftrightarrow (ii)$ by changing the roles of $T$ and $S$.
\end{proof}
Now we present a characterization concerning compact operators. We need the following lemma.
\begin{lemma}\cite[Theorem 3.10]{M.Z}\label{lm.28}
Let $T\in \mathbb{B}(\mathscr{H})$. Then the following statements are equivalent:
\begin{itemize}
\item[(i)] $T\parallel I$.
\item[(ii)] There exist a sequence of unit vectors $\{\xi_n\}$ in $\mathscr{H}$ and $\lambda\in\mathbb{T}$ such that
$$\lim_{n\rightarrow\infty} \Big\|T\xi_n-\lambda\|T\|\xi_n\Big\| = 0.$$
\end{itemize}
\end{lemma}
\begin{theorem}\label{th.29}
Let $T\in \mathbb{K}(\mathscr{H})$. Then the following statements are equivalent:
\begin{itemize}
\item[(i)] $T\parallel I$.
\item[(ii)] $\lambda \|T\|$ is an eigenvalue of $T$ for some $\lambda\in\mathbb{T}$.
\item[(iii)] $|T|\parallel I$.
\end{itemize}
\end{theorem}
\begin{proof}
(i)$\Longrightarrow$(ii) Suppose that $T\parallel I$ holds. By Lemma \ref{lm.28}, there exists a sequence of unit vectors $\{\xi_n\}$ in $\mathscr{H}$ such that
\begin{align}\label{id.30}
\lim_{n\rightarrow\infty} \Big\|T\xi_n-\lambda\|T\|\xi_n\Big\| = 0.
\end{align}
Since $T$ is compact, there exist a subsequence $\{\xi_{n_k}\}$ and $\xi_0\in\mathscr{H}$ such that
\begin{align}\label{id.31}
\lim_{k\rightarrow\infty} T\xi_{n_k} = \xi_0.
\end{align}
Passing to the limit in (\ref{id.30}) we obtain $\Big\|\xi_0-\lim_{k\rightarrow\infty} \lambda \|T\|\xi_{n_k}\Big\| = 0$, or equivalently, $\lim_{k\rightarrow\infty} \xi_{n_k} = \frac{\overline{\lambda}}{\|T\|}\xi_0.$
Since $T$ is bounded, $$\lim_{k\rightarrow\infty} T\xi_{n_k} = \frac{\overline{\lambda}}{\|T\|}T\xi_0.$$
Therefore $\frac{\overline{\lambda}}{\|T\|}T\xi_0 = \xi_0$, or equivalently, $T\xi_0 = \lambda\|T\|\xi_0$. Thus $\lambda \|T\|$ is an eigenvalue of $T$.\\
(ii)$\Longrightarrow$(iii) Suppose that (ii) holds. Then there exists a unit vector $\xi_0\in\mathscr{H}$ such that $T\xi_0 = \lambda\|T\|\xi_0$. We also have $$\Big\||T|\xi_0\Big\| = \|T\xi_0\| = \Big\|\lambda\|T\|\xi_0\Big\| = \|T\| = \Big\||T|\Big\|,$$
we get $r(|T|) = \Big\||T|\Big\|$. By the equivalence $(i) \Leftrightarrow (iv)$ of Lemma \ref{lemma.280}, we have $|T|\parallel I$.\\
(iii)$\Longrightarrow$(i) Let $|T|\parallel I$. Since $T$ is compact we deduce that $|T|$ is compact. So, by the implication $(i) \Rightarrow (ii)$, there is $\lambda\in\mathbb{T}$ such that $\lambda \Big\||T|\Big\|$ is an eigenvalue of $|T|$. Therefore there exists a unit vector $\xi_0\in\mathscr{H}$ such that $|T|\xi_0 = \lambda\Big\||T|\Big\|\xi_0 = \lambda\|T\|\xi_0$. Hence
$$\|T\xi_0\| = \sqrt{[T^*T\xi_0, \xi_0]} = \sqrt{[|T|\xi_0, |T|\xi_0]} = \sqrt{[\lambda\|T\|\xi_0, \lambda\|T\|\xi_0]} = \|T\|,$$
whence $\|T\| \leq \sup\{|\gamma|: \, \gamma\in\sigma(T)\} = r(T) \leq \|T\|$. Thus we get $r(T) = \|T\|$. By the equivalence $(i) \Leftrightarrow (iv)$ of Lemma \ref{lemma.280}, we reach $T\parallel I$.
\end{proof}
\begin{remark} \label{re.31.1}
Notice that if $T\in \mathbb{B}(\mathscr{H})$ and $\lambda \|T\|$ is an eigenvalue of $T$ for some $\lambda\in\mathbb{T}$, then there exists a vector $x\in\mathscr{H}\smallsetminus\{0\}$ such that $Tx = \lambda\|T\|x$. Thus we have $\|Tx\| = \|T\|$. So, $r(T) = \|T\|$. By the equivalence $(i) \Leftrightarrow (iv)$ of Lemma \ref{lemma.280}, we get $T\parallel I$. On the other side, if we consider the shift operator $T\,:\ell^2\longrightarrow \ell^2$ defined by
$T(\xi_1, \xi_2, \xi_3, \cdots) = (0, \xi_1, \xi_2, \xi_3, \cdots)$, then $r(T) = 1 = \|T\|$. By the equivalence $(i) \Leftrightarrow (iv)$ of Lemma \ref{lemma.280} we get $T \parallel I$. Also, it is easily seen that $T$ has no eigenvalues and $T$ is not compact.
This shows that the condition of compactness in the implication $(ii)\Rightarrow (i)$ of Theorem \ref{th.29} is essential.
\end{remark}
Recall that if $\zeta$ and $\eta$ are elements of a Hilbert space $(\mathscr{H}, [., .])$, then $r(\zeta\otimes\eta) = |[\zeta, \eta]|$, where $\zeta\otimes\eta$ is the rank one operator defined by $(\zeta\otimes\eta)(\xi)=[\xi, \eta]\zeta \,\,(\xi\in\mathscr{H})$. As an immediate consequence of Theorem \ref{th.29} and Lemma \ref{lemma.280}, we get a characterization of the linear dependence of two elements of a Hilbert space.
\begin{corollary}\label{cr.35.11}
Let $\eta, \xi \in \mathscr{H}$. Then the following statements are equivalent:
\begin{itemize}
\item[(i)] $\eta\parallel\xi$.
\item[(ii)] $\zeta\otimes\eta\parallel \zeta\otimes\xi$ \quad for all $\zeta\in \mathscr{H}$.
\item[(iii)] $\eta\otimes\xi\parallel I$.
\item[(iv)] $\lambda \|\eta\|\,\|\xi\|$ is an eigenvalue of $\eta\otimes\xi$ for some $\lambda\in\mathbb{T}$.
\end{itemize}
\end{corollary}
\begin{proof}
Suppose that $\zeta\in \mathscr{H}$. We may assume that $\zeta, \eta, \xi \neq 0$. Thus we have
\begin{align*}
\zeta\otimes\eta\parallel \zeta\otimes\xi & \Longleftrightarrow r((\zeta\otimes\eta)^\ast\zeta\otimes\xi) = \|(\zeta\otimes\eta)^\ast\zeta\otimes\xi\| = \|\zeta\otimes\eta\|\,\|\zeta\otimes\xi\|
\\&\hspace{4cm}(\mbox{by the equivalence}\,\, (i) \Leftrightarrow (iv) \,\mbox{of Lemma} \,\ref{lemma.280})
\\& \Longleftrightarrow r([\zeta, \zeta]\eta\otimes\xi) = \|[\zeta, \zeta]\eta\otimes\xi\| = \|\zeta\|\,\|\eta\|\,\|\zeta\|\,\|\xi\|
\\& \Longleftrightarrow \|\zeta\|^2r(\eta\otimes\xi) = \|\zeta\|^2\,\|\eta\|\,\|\xi\|
\\& \Longleftrightarrow |[\eta, \xi]| =\|\eta\|\,\|\xi\|
\\& \Longleftrightarrow \eta, \xi \,\mbox{are linearly dependent}
\\& \Longleftrightarrow \eta\parallel\xi.
\end{align*}
The proofs of the other equivalences are similar, so we omit them.
\end{proof}

\textbf{Acknowledgement.}  The second author (corresponding author) would like to thank the Department of Mathematics and Computer Science at Karlstad University in Sweden as well as the Senior Associate scheme of the Abdus Salam International Centre for Theoretical Physics (ICTP).

\bibliographystyle{amsplain}

\end{document}